\documentclass{amsart}
\usepackage{graphics}

\newcommand{\Q}{\mathbb{Q}}
\newcommand{\C}{\mathbb{C}}
\newcommand{\D}{\mathbb{D}}

\newcommand{\R}{\mathbb{R}}
\newcommand{\N}{\mathbb{N}}

\newcommand{\dom}{\operatorname{dom}}

\newcommand{\diam}{\operatorname{diam}}

\renewcommand{\Re}{\operatorname{Re}}
\renewcommand{\Im}{\operatorname{Im}}

\newtheorem{theorem}{Theorem}[section]
\newtheorem{lemma}[theorem]{Lemma}

\theoremstyle{definition}
\newtheorem{definition}[theorem]{Definition}

\theoremstyle{theorem}

\theoremstyle{theorem}
\newtheorem{proposition}[theorem]{Proposition}

\theoremstyle{theorem}

\theoremstyle{theorem}

\theoremstyle{definition}

\theoremstyle{theorem}

\numberwithin{equation}{section}

\begin{document}
\title[Computing boundary extensions part 2]{Computing boundary extensions of conformal maps part 2}

\author{Timothy H. McNicholl}
\address{Department of Mathematics\\
Iowa State University\\
Ames, Iowa 50011 USA
}
\email{mcnichol@iastate.edu}

\begin{abstract}    
It is shown that there is a computable and conformal map on the unit disk that has a computable boundary extension even though its range does not have a computable boundary connectivity function.
\end{abstract}

\maketitle

\section{Introduction}\label{sec:INTRO}

We continue the investigation in \cite{McNicholl.2011.b} of the information required to compute the boundary extension of a computable conformal map.  In \cite{McNicholl.2011.b}, it is shown that local connectivity provides sufficient information.  Here, were show that it does not provide \emph{necessary} information.

Local connectivity information for the boundary of a domain $D$ is quantified by means of a \emph{boundary connectivity function}.  This is a function $g : \N \rightarrow \N$ with the property that whenever $k \in \N$ and $p,q$ are boundary points of $D$ so that $0 < |p - q| \leq 2^{-g(k)}$, the boundary of $D$ includes an arc from $p$ to $q$ whose diameter is smaller than $2^{-k}$.  The case for using boundary connectivity functions to represent the local connectivity of the boundary of $D$ is made in \cite{McNicholl.2011.b}.  See also \cite{Daniel.McNicholl.2012}.  

 A conformal map of the unit disk onto a domain $D$ has a boundary extension if and only if $D$ is bounded and has a boundary connectivity function.  In \cite{McNicholl.2011.b}, it is shown that if $\phi$ is a computable and conformal map of the unit disk onto a bounded domain $D$ that has a computable boundary connectivity function, then the boundary extension of $\phi$ is computable.  Here, we complement this result by proving the following.

\begin{theorem}\label{thm:MAIN}
There is a conformal map on the unit disk that has a computable boundary extension even though its range does not have a computable boundary connectivity function.
\end{theorem}

Theorem \ref{thm:MAIN} strengthens a result from \cite{Couch.Daniel.McNicholl.2012} on the computability of Peano continua.  Namely, there is a computable function on the unit interval whose range does not have a computable \emph{modulus of local connectivity}.  This result says that the converse of the Hahn-Mazurkiewicz Theorem is not effective.  By a modulus of local connectivity for a subset of $\R^n$ $X$ we mean a function $g : \N \rightarrow \N$ so that whenever $k \in \N$ and $p,q \in X$ are such that $0 < d(p,q) \leq 2^{-g(k)}$, $X$ includes an arc from $p$ to $q$ whose diameter is smaller than $2^{-k}$.  Hence, a boundary connectivity function for a domain $D$ is a modulus of local connectivity for the boundary of $D$.   

We build the conformal map in Theorem \ref{thm:MAIN} by first constructing its range, $D$.  To construct $D$, we first form its boundary which will consist of a connected union of countably many line segments.  We fix a computably enumerable but incomputable set of natural numbers $A$, and choose these line segments so as to encode $A$ into every boundary connectivity function for $D$ (see Theorem \ref{thm:NOT.ELC}).  In particular, by taking $A$ to be the Halting Set, we can encode all computably enumerable sets into every boundary connectivity function for $D$.  That is, every boundary connectivity function for $D$ computes every computably enumerable set.  

We construct $D$ and its boundary so that $D$ is computably open and its boundary is computably closed.  It then follows from the Effective Riemann Mapping Theorem \cite{Hertling.1999} (see also \cite{Bishop.Bridges.1985}) that there is a computable and conformal map of the unit disk onto $D$.  Let $\phi$ denote such a map.  

The boundary of $D$ is constructed so that $\phi$ has a boundary extension.  The challenge then is to compute the boundary extension of $\phi$ without local connectivity information.  Our strategy is to carefully modify the approach in \cite{McNicholl.2011.b}; in particular, many of these results do not make use of the boundary connectivity function and so they apply equally well to the situation considered here.
Most importantly, the concept of a crosscut that recognizably bounds the value of $\phi$ on a unimodular $\zeta$ can be employed here.  These crosscuts (which will be defined in Section \ref{sec:BACKGROUND}) have the property that if $\tau$ is a sufficiently small boundary arc that joins its endpoints and contains no point of $D$, then $\phi(\zeta) \in \tau$.  In \cite{McNicholl.2011.b}, a boundary connectivity function is then used to form an upper bound on the diameter of $\tau$ from the diameter of the crosscut.  We show in Subsection \ref{subset:CLASSICAL} of Section \ref{sec:BACKGROUND} that without a boundary connectivity function we can still make progress by singling out a class of \emph{acceptable} crosscuts from those that recognizably bound the value of $\phi$ on $\zeta$.  In particular, our domain $D$ is constructed so that the endpoints of an acceptable crosscut $C$ can be joined by an arc $\tau$ that contains no point of $D$ and that consists of a single horizontal line segment, a single vertical line segment, or a union of a horizontal and vertical line segment that meet at a right angle.  This allows us to effectively bound the diameter of $C \cup \tau$.  We then show that $\phi(\zeta) \in \tau$.   

\section{Background and summary of prior results}\label{sec:BACKGROUND}

For information on classical computability notions such as computable enumerability and Turing reducibility we refer the reader to \cite{Cooper.2004}.

Suppose $f$ is a function that maps complex numbers to complex numbers.  We say that $f$ is \emph{computable} if there is an algorithm $P$ that satisfies the following three criteria.
\begin{itemize}
	\item \bf Approximation:\rm\ Whenever $P$ is given an open rational rectangle as input, it either does not halt or produces an open rational rectangle as output.  (Here, the input rectangle is regarded as an approximation of a $z \in \dom(f)$ and the output rectangle is regarded as an approximation of $f(z)$.)

	\item \bf Correctness:\rm\ Whenever $P$ halts on an open rational rectangle $R$, the rectangle it outputs contains $f(z)$ for each $z \in R \cap \dom(f)$. 
	
	\item \bf Convergence:\rm\ Suppose $U$ is a neighborhood of a point $z \in \dom(f)$ and that $V$ is a neighborhood of $f(z)$.  Then, there is an open rational rectangle $R$ such that $R$ contains $z$, $R$ is included in $U$, and when $R$ is put into $P$, $P$ produces a rational rectangle that is included in $V$.
\end{itemize}
A comprehensive treatment of the computability of functions on continuous domains can be found in \cite{Weihrauch.2000}.  See also \cite{Turing.1937}, \cite{Grzegorczyk.1957}, \cite{Lacombe.1955.a}, \cite{Lacombe.1955.b}, \cite{Brattka.Weihrauch.1999}, \cite{Pour-El.Richards.1989}, and \cite{Braverman.Cook.2006}.  

Let $\D$ denote the unit disk; that is the open disk with center $0$ and radius $1$.  The \emph{unit circle} is the boundary of $\D$.   Our approach to computing boundary extensions is based on the following notion and result from \cite{McNicholl.2011.b}.

\begin{definition}\label{def:STRONGLY.COMPUTABLE}
Suppose $f$ is a function that maps complex numbers to complex numbers and is defined at every point on the unit circle.  We say that $f$ is \emph{strongly computable on the unit circle} if there is an algorithm $P$ with the following properties.
\begin{itemize}
	\item  \bf Approximation:\rm\ Whenever an open rational rectangle is input to $P$, $P$ either does not halt or outputs an open rational rectangle.
	\item  \bf Strong Correctness:\rm\ If $P$ outputs a rational rectangle $R_1$ on input $R$, then $f(z) \in R_1$ whenever $z \in R \cap \dom(f)$.
	\item \bf Convergence:\rm\ If $U$ is a neighborhood of a unimodular point $\zeta$, and if $V$ is a neighborhood of $f(\zeta)$, then $\zeta$ belongs to an open rational rectangle $R \subseteq U$ so that $P$ halts on input $R$ and produces a rational rectangle that is contained in $V$.
\end{itemize}
\end{definition}

\begin{proposition}\label{prop:STRONGLY.COMPUTABLE}\label{prop:STRONGLY.COMP}
Suppose $f : \overline{\D} \rightarrow \C$.  Then, $f$ is computable if and only if $f$ is both computable on the unit disk and strongly computable on the unit circle.
\end{proposition}

When $U$ is an open subset of the plane, let $R(U)$ denote the set of all closed rational rectangles that are included in $U$.  When $C$ is a closed subset of the plane, let $R(C)$ denote the set of all open rational rectangles that contain at least one point of $C$.  Whether $X$ is open or closed, the set $R(X)$ completely identifies $X$.  That is, $R(X) = R(X')$ if and only if $X = X'$.

Let us call an open subset of the plane $U$ \emph{computable} if $R(U)$ is computably enumerable.  That is, if the elements of $R(U)$ can be arranged into a sequence $\{R_n\}_{n \in \N}$ in such a way that there is an algorithm that computes $R_n$ from $n$ for every $n \in \N$.  Intuitively, as such an enumeration is run, it provides more and more information about what is in the set.  
We similarly define what it means for a closed subset of the plane to be computable.  Again, by enumerating the rational rectangles that contain at least one point of a closed set $C$ we obtain more and more information about what is in the set.  

Throughout the rest of this section, $\phi$ denotes a conformal map of the unit disk onto a domain $D$.  We assume that $\phi$ has a boundary extension which we denote by $\phi$ as well. 

Suppose $C$ is an arc in $\overline{D}$.  If the only points of $C$ that lie on the boundary of $D$ are the endpoints of $C$, then $C$ is called a \emph{crosscut} of $D$.  If $C$ is a crosscut of $D$, then $D - C$ has exactly two connected components.  These components are called the \emph{sides} of $C$.   
When $C$ is a crosscut of $D$ that omits $\phi(0)$, let $C^-$ denote the side of $C$ that contains $\phi(0)$ and let $C^+$ denote the other side.  

Let $D_r(z)$ denote the open disk with center $z$ and radius $r$.  When $0 < r < 1$ and $|\zeta| = 1$, let $A_{r, \zeta}$ denote the image of $\phi$ on $\overline{\D} \cap \partial D_r(\zeta)$.  The class of crosscuts that recognizably bound the value of $\phi$ on $\zeta$ is defined as follows.

\begin{definition}\label{def:RECOGNIZABLE}  Suppose $|\zeta| =1$.  Let $C$ be a crosscut of $D$.  We say that $C$ \emph{recognizably bounds the value of $\phi$ on $\zeta$} if there are rational numbers $r_0,s_0$ such that the following hold.  
\begin{enumerate}
	\item $0 < r_0 < s_0 < 1/2$.\label{def:RECOGNIZABLE.1}
	
	\item $\phi((1 - s_0)\zeta) \in C$.\label{def:RECOGNIZABLE.2}
	
	\item $C \cap A_{s_0, \zeta}$ is connected, and $C \cap A_{s_0, \zeta}^+$ has two connected components.\label{def:RECOGNIZABLE.3}
	
	\item $|\phi(t\zeta) - z| >  m(s_0, N_0, r_0)$ whenever $z \in \overline{A_{s_0, \zeta}^+ \cap C}$ and $1 - s_0 \leq t \leq 1 - r_0$.\label{def:RECOGNIZABLE.four}
\end{enumerate}
We say that $(r_0, s_0)$ \emph{witnesses} that $C$ recognizably bounds the value of $\phi$ on $\zeta$.
\end{definition}

Let $\mathcal{C}_{(s_0, r_0, \zeta)}$ denote the set of all crosscuts $C$ such that $(s_0, r_0)$ witnesses that $C$ recognizably bounds the value of $\phi$ on $\zeta$.  

The fundamental facts about these crosscuts are the following two theorems which are proved in \cite{McNicholl.2011.b}.

\begin{theorem}\label{thm:EXISTENCE}  Suppose $|\zeta| = 1$.
Then, there are crosscuts of arbitrarily small diameter that recognizably bound the value of $\phi$ on $\zeta$.  That is, for every $\epsilon > 0$ there is a crosscut that recognizably bounds the value of $\phi$ on $\zeta$ and whose diameter is smaller than $\epsilon$.
\end{theorem}

\begin{theorem}\label{thm:BOUNDING.SIDE}
Suppose $(r_0, s_0)$ witnesses that $C$ recognizably bounds the value of $\phi$ on $\zeta$.  Then, $A_{r_0, \zeta}^+ \subseteq C^+$. 
\end{theorem}

We now discuss approximation of crosscuts.  

A finite sequence of sets $(S_1, \ldots, S_n)$ is a \emph{chain} if $S_j \cap S_{j+1} \neq \emptyset$ whenever $1 \leq j < n$.  In addition, $(S_1, \ldots, S_n)$ is a \emph{simple chain} if $S_j \cap S_k \neq \emptyset$ only when $|j - k| = 1$.   We then define a \emph{wad} to be a union of a chain of open rational boxes and an \emph{approximate arc} to be a simple chain of wads.  

When $A_1, \ldots, A_n$ are subarcs of an arc $A$, we write $A = A_1 + \ldots + A_n$ if $A_{j+1}$ contains exactly one point of $A_j$ whenever $1 \leq j < n$.  
An approximate arc $(w_1, \ldots, w_n)$ \emph{approximates} an arc $A$ if $A$ can be decomposed into a sum $A = A_1 + \ldots + A_n$ such that 
$A_j \subseteq w_j$ for all $j$.  Equivalently, if there are numbers $0 = t_0 < \ldots < t_n = 1$ such that $A$ maps each number in $[t_{j-1}, t_j]$ into $w_j$.  The largest diameter of a wad $w_j$ will be referred to as the \emph{error} in this approximation. 

We define an \emph{approximate crosscut} of $D$ to be an approximate arc $(w_1, \ldots, w_n)$ such that 
\begin{itemize}
	\item $\overline{w_j} \subseteq D$ when $1 < j < n$, and 
	
	\item $w_j \cap \partial D \neq \emptyset$ if $j = 1, n$.
\end{itemize}

We will use the following Lemma from \cite{McNicholl.2011.b}.

\begin{lemma}\label{lm:UNIMODULAR}
Suppose $|\zeta| = 1$ and $(s_0, r_0)$ witnesses that a crosscut $C$ recognizably bounds the value of $\phi$ on $\zeta$.  Suppose $C$ is approximated by $(w_1, \ldots, w_n)$.   Then, whenever $\zeta'$ is a unimodular point that is sufficiently close to $\zeta$, $(w_1, \ldots, w_n)$ approximates a crosscut $C'$ such that $(s_0, r_0)$ witnesses that $C'$ recognizably bounds the value of $\phi$ on $\zeta'$.  
\end{lemma}

The following definition is also from \cite{McNicholl.2011.b}.

\begin{definition}\label{def:DESCRIBES}
Suppose that $\mathcal{C}$ is a set of crosscuts of $D$ and that $\mathcal{A}$ is a set of approximate crosscuts.  We say that $\mathcal{A}$ \emph{describes} $\mathcal{C}$ if the following two conditions are met.
\begin{enumerate}
	\item Every approximate crosscut in $\mathcal{A}$ approximates a crosscut in $\mathcal{C}$.
	
	\item Every crosscut in $\mathcal{C}$ can be approximated with arbitrarily small error by an approximate crosscut in $\mathcal{A}$.  That is, if $C$ is a crosscut in $\mathcal{C}$, and if $\epsilon > 0$, then $C$ is approximated by an approximate crosscut in $\mathcal{A}$ with error smaller than $\epsilon$.
\end{enumerate}
\end{definition}

The following is a porism of the proof of Theorem 3.7 of \cite{McNicholl.2011.b}.

\begin{theorem}\label{thm:ALPHA}\label{thm:DESCRIBES.CE}
For all rational numbers $s_0, r_0$ and unimodular $\zeta$, there is a  set of approximate crosscuts $\mathcal{A}_{(s_0, r_0, \zeta)}$ that describes $\mathcal{C}_{(s_0, r_0, \zeta)}$.   Furthermore, $\mathcal{A}_{(s_0, r_0, \zeta)}$ is computably enumerable uniformly in $(s_0, r_0, \zeta)$.
\end{theorem}

When $z,w$ are distinct points in the plane, let $[z,w]$ denote the line segment from $z$ to $w$; let $[z,z] = \{z\}$.  Let $\diam(C)$ denote the diameter of $C$.  

\section{The construction of the domain $D$}\label{sec:CONSTRUCTION}

We begin by constructing a subset of the plane $X$.  This set will be the boundary of $D$ and consists of a union of countably many line segments.  We will choose these line segments so as to encode a computably enumerable but incomputable set into any modulus of local connectivity for $X$.  

Let $A$ be a computably enumerable but incomputable set of natural numbers.  Since $A$ is computably enumerable, it is the domain of a Turing computable function that maps natural numbers to natural numbers.  Let $f$ be such a function, and let $M$ be a Turing machine that computes $f$.  We say that a number $n$ \emph{enters $A$ at stage $s$} if $M$ halts on input $n$ after exactly $s$ steps; in this case we write $n \in A_{\mbox{at $s$}}$.

Let:
\begin{eqnarray*}
\nu_0 & = & 0\\
\nu_1 & = & i\\
\nu_2 & = & 1 + i\\
\nu_3 & = & 1\\
\end{eqnarray*}
Let:
\begin{eqnarray*}
\nu_{3j + 4} & = & \left\{
\begin{array}{cc}
2^{-(j+1)} + 2^{-(j + 3 + s)} & \mbox{if $j \in A_{\mbox{at $s$}}$}\\
2^{-(j+1)} & \mbox{if $j \not \in A$}\\
\end{array}
\right.\\
\ & \ & \ \\
\nu_{3j + 5} & = & 2^{-(j+1)}(1 + i)\\
\ & \ & \ \\
\nu_{3j + 6} & = & \left\{
\begin{array}{cc}
2^{-(j+1)} - 2^{-(j + 3 + s)} & \mbox{if $j \in A_{\mbox{at $s$}}$}\\
2^{-(j+1)} & \mbox{if $j \not \in A$}\\
\end{array}
\right.
\end{eqnarray*}
Let
\[
X = \bigcup_{n \in \N} [\nu_n, \nu_{n+1}].
\]
The set $\C - X$ has two connected components, only one of which is bounded.  Let $D$ denote the bounded component.  

We call $\nu_0$, $\nu_1$, $\ldots$ the \emph{vertices} of $D$.  Fix $j \in \N$.  If $\nu_{3j + 4} \neq \nu_{3j + 6}$, then we call $[\nu_{3j+4}, \nu_{3j + 5}] \cup [\nu_{3j + 5}, \nu_{3j + 6}]$ a \emph{tent}; otherwise we refer to it as a \emph{spike}.  See Figures \ref{fig:n_in_A} and \ref{fig:n_not_in_A}.  The \emph{significant constituents} of the boundary of $D$ consist of the tents, spikes, and the line segments in the boundary of $D$ that are included in a side of the unit square. 
\begin{figure}[!h]
\resizebox{3in}{3in}{\includegraphics{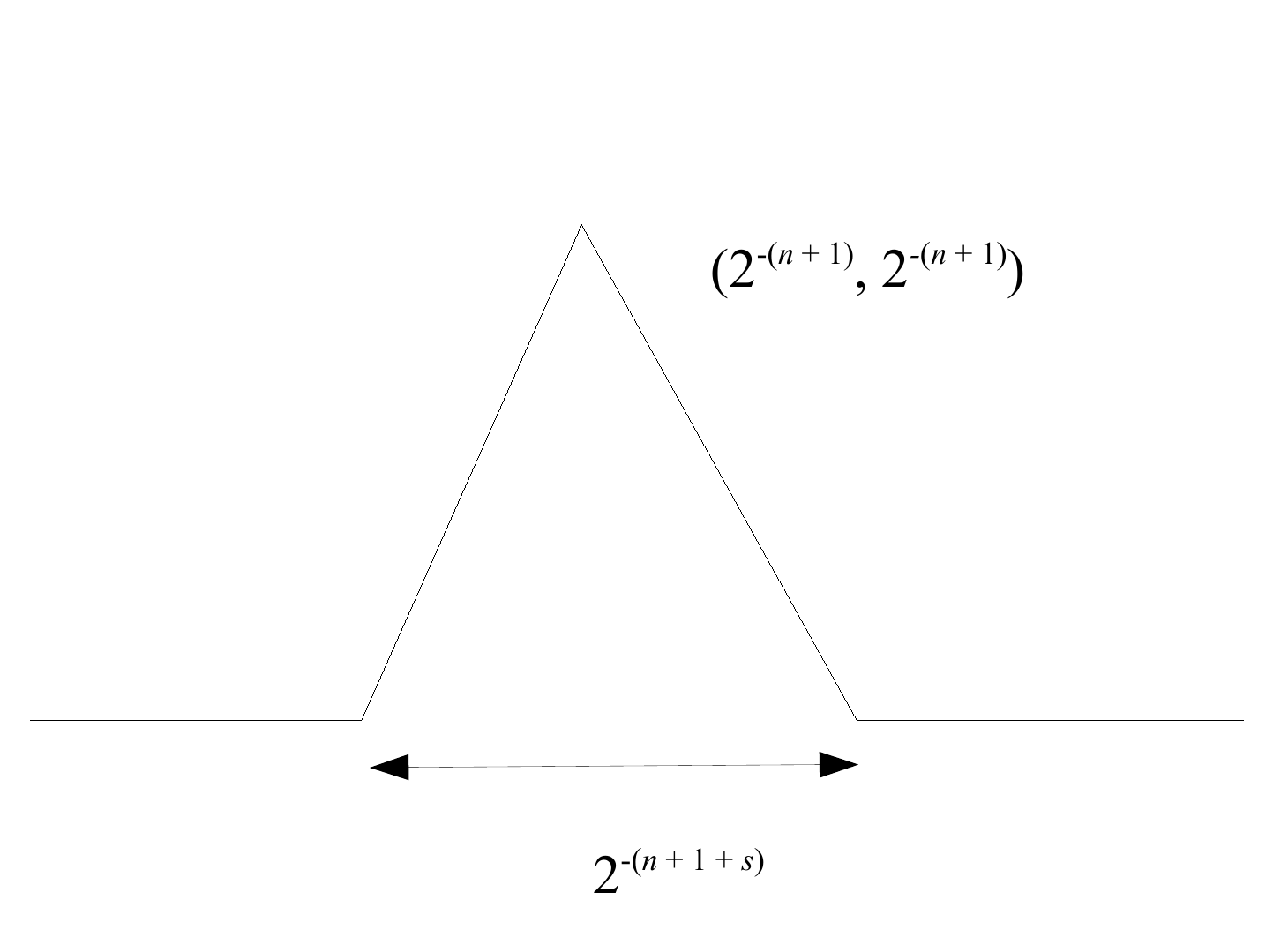}}
\caption{Construction of $X$: the case where $n \in A$}\label{fig:n_in_A}
\end{figure}
 \begin{figure}
\resizebox{3.5in}{3in}{\includegraphics{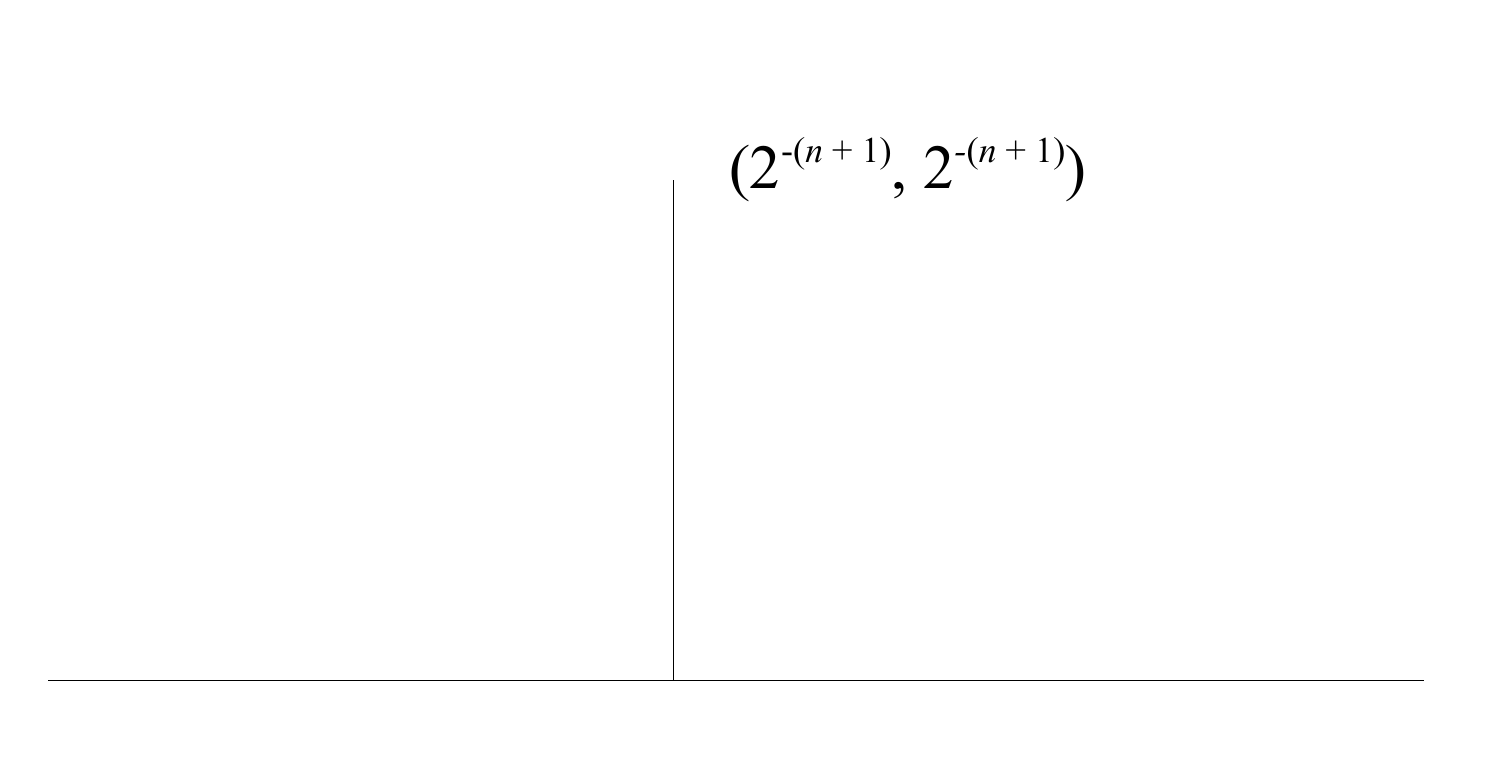}}
\caption{Construction of $X$: the case where $n \not \in A$}\label{fig:n_not_in_A}
\end{figure}

%

The set $D$ is simply connected and $X$ is the boundary of $D$.  It follows from the construction that $D$ is computably open and $X$ is computably closed.  So, by the Effective Riemann Mapping Theorem of Hertling \cite{Hertling.1999}, there is a computable and conformal map $\phi$ of $\D$ onto $D$.  The set $X$ is locally connected, so $\phi$ has a boundary extension (see Pommerenke \cite{Pommerenke.1992}).  We denote this boundary extension by $\phi$ as well.

\section{The verification of the construction}\label{sec:VERIFICATION}

We divide the verification into three layers: the classical world, wherein we have full access to all the objects and methods of classical mathematics, the approximate world wherein we can only view approximations of classical objects, and the computable world in which only computable operations with approximations are permitted.

\subsection{The classical world}\label{subset:CLASSICAL}.

Our first goal is to define the class of acceptable crosscuts.  

Suppose that $z_1,z_2$ are points in the plane.  Each of the sets
\[
[z_1, \Re(z_1) + \Im(z_2) i] \cup [\Re(z_1) + \Im(z_2) i, z_2], 
[z_1, \Re(z_2) + \Im(z_1) i] \cup [\Re(z_2) + \Im(z_1) i, z_2]
\]
will be called a \emph{taxicab arc} from $z_1$ to $z_2$.  So, each taxicab arc is either a horizontal line segment, a vertical line segment, or an arc that consists of a  union of a horizontal line segment and a vertical line segment that meet at a right angle.

\begin{definition}\label{def:ACCEPTABLY.PLACED}
Suppose $z_1$ and $z_2$ are distinct boundary points of $D$.  We say that $z_1, z_2$ are \emph{acceptably placed} if neither is a vertex and if at least one taxicab arc from $z_1$ to $z_2$ contains no point of $D$.
\end{definition}

\begin{definition}\label{def:ACCEPTABLE}
Let $C$ be a crosscut of $D$.  We say that $C$ is \emph{acceptable} if its endpoints are acceptably placed and its diameter is smaller than 
\[
\frac{|\phi(0) - \phi(\zeta/2)|}{1 + \sqrt{2}}
\]
for each unimodular $\zeta$.
\end{definition}
%

It follows that if $z$ is a boundary point of $D$ that is not a vertex, then $z$ belongs to exactly one significant constituent of the boundary of $D$.

\begin{lemma}\label{lm:ACCEPTABLY.PLACED}
Let $z,z'$ be distinct boundary points of $D$, and suppose that neither of these points is a vertex.  Let $\sigma$ be the significant constituent of the boundary of $D$ that contains $z$, and let $\sigma'$ be the significant constituent of the boundary of $D$ that contains $z'$.  Then, $z, z'$ are acceptably placed if and only if $w, w'$ are acceptably placed whenever $w \in \sigma$ and $w' \in \sigma'$ are distinct and are not vertices.
\end{lemma}

\begin{proof}
Suppose $z,z'$ are acceptably placed.   Let $w$ be point in $\sigma$ that is not a vertex, and let $w'$ be a point $\sigma'$ that is not a vertex.  Suppose $w \neq w'$.  We show that $w,w'$ are acceptably placed.  That is, at least one taxicab arc from $w$ to $w'$ contains no point of $D$.

If $\sigma = \sigma'$, then it follows by inspection that at least one taxicab arc from $w$ to $w'$ contains no point of $D$.  So, suppose $\sigma \neq \sigma'$.

Suppose one of $\sigma, \sigma'$ is a side of the unit square.  Since $z,z'$ are acceptably placed, it follows that an adjacent side includes the other significant constituent.  It then follows that $w,w'$ are acceptably placed.  So, suppose that neither of $\sigma$, $\sigma'$ is a side of the unit square.  
If the lower side of the unit square includes both $\sigma$ and $\sigma'$, then $w,w'$ are acceptably placed.  So, suppose one of $\sigma$, $\sigma'$ is not included in the lower side of the unit square; that is it is either a tent or a spike.  Since $z,z'$ are acceptably placed, it follows that the other one of these significant constituents must be included in the lower side.  Thus, $w$ and $w'$ are acceptably placed.
\end{proof}

We now show that acceptable crosscuts that recognizably bound the value of $\phi$ on $\zeta$ allow us to estimate the value of $\phi$ on all points near $\zeta$ without the use of a boundary connectivity function.  This is based on the following two results.

 \begin{proposition}\label{prop:INTERIOR}
 Let $C$ be a crosscut of $D$.  Let $\tau$ be an arc that joins the endpoints of $C$ and that contains no point of $D$.  Then, the interior of $\tau \cup C$ includes exactly one side of $C$.
 \end{proposition}
 
 \begin{proof}
 Let $J = C \cup \tau$.  
 
 We first claim that if the interior of $J$ contains one point of a side of $C$, then it must include all of that side.  For, let $S$ be a side of $C$, and suppose the interior of $J$ contains a point of $S$.  Let $q_1$ denote such a point.  Let $q_2$ be any other point of $S$.  Since $S$ is open and connected, it includes an arc $\sigma$ from $q_1$ to $q_2$.  Since $D$ includes $\sigma$, $\tau$ contains no point of $\sigma$.  Since $S$ includes $\sigma$, $\sigma$ contains no point of $C$.  Thus, $\sigma$ never crosses $J$.  So, since the interior of $J$ is connected, $q_2$ belongs to the interior of $J$ as well.
 
 We now show that the interior of $J$ includes \emph{at least} one side of $C$.  For, let $p \in C \cap D$.  Since $D$ is open, there is a positive number $\epsilon$ so that $D$ includes $D_\epsilon(p)$.  Since $p$ belongs to the boundary of the interior of $J$, $D_\epsilon(p)$ contains a point of the interior of $J$.  Let $q$ denote such a point.  Since $q$ belongs to the interior of $J$, $q$ does not belong to $C$.  But, since $D$ includes $D_\epsilon(p)$, $q$ belongs to $D$.  Thus, $q$ must belong to a side of $C$.  By what has just been shown, the interior of $J$ includes all of this side.
 
 Finally, we claim that the interior of $J$ does not include both sides of $C$.  For, by way of contradiction, suppose otherwise.  Then, the interior of $J$ includes $D - C$.  Let $p \in C \cap D$.  Let $U$ denote the interior of $J$.  Thus, $\overline{U} = U \cup J$.  In addition, $p \in \partial \overline{U}$.  Since $p \in D$, $D_\epsilon(p) \subseteq D$ for some $\epsilon > 0$.  So, $D_\epsilon(p)$ contains a point of the complement of $\overline{U}$, $q$.  There is a positive number $\delta$ so that $D_\delta(q) \subseteq D_\epsilon(p) \cap (\C - \overline{U})$.  Thus, $D_\delta(q) \subseteq D - C$.   But, the interior of $J$ includes $D - C$ so this is a contradiction.
 \end{proof}
 
 \begin{theorem}\label{thm:WITNESSES} 
 If there is an acceptable crosscut $C$ so that $(s_0, r_0)$ witnesses that $C$ recognizably bounds the value of $\phi$ on $\zeta$, then the diameter of 
 $C$ is at least as large as $(1 + \sqrt{2})^{-1}|\phi(z_1) - \phi(z_2)|$ for all $z_1, z_2 \in D_{r_0}(\zeta) \cap \D$.
 \end{theorem}
 
 \begin{proof}
 Suppose $C$ is acceptable and that $(s_0, r_0)$ witnesses that $C$ recognizably bounds the value of $\phi$ on $\zeta$.
 Let $p,q$ be the endpoints of $C$.  Since $C$ is acceptable, there is a taxicab arc from $p$ to $q$ that contains no point of $D$; let $\tau$ be such an arc.  Let $J = C \cup \tau$.  Note that the diameter of $\tau$ is at most $\sqrt{2}|p - q|$.  
 Thus, the diameter of $J$ is at most $(1 + \sqrt{2})\diam(C)$.  
 
 We claim that the interior of $J$ does not include $C^-$.  For, by the definition of `acceptable' (Definition \ref{def:ACCEPTABLE}), 
 \[
 (1 + \sqrt{2}) \diam(C) < |\phi(0) - \phi(\zeta/2)|.
 \]
Thus, the diameter of the interior of $J$ is smaller than $|\phi(0) - \phi(\zeta/2)|$.  Since $s_0 < 1/2$ (by definition of `recognizably bounds'; see Definition \ref{def:RECOGNIZABLE}), $C^-$ contains $\phi(0)$ and $\phi(\zeta/2)$.  Thus, the interior of $J$ can not include $C^-$.  
 
 Now, by Proposition \ref{prop:INTERIOR}, the interior of $J$ includes exactly one side of $C$.  Thus, it must include $C^+$.  So, by Theorem \ref{thm:BOUNDING.SIDE}, it includes $A_{r_0, \zeta}^+$.  This means that $|\phi(z_1) - \phi(z_2)| \leq (1 + \sqrt{2}) \diam(C)$ for all $z_1, z_2 \in D_{r_0}(\zeta) \cap \D$.
 \end{proof}

Since we will use acceptable crosscuts to estimate boundary values, we seek to prove that they not only exist but that arbitrarily small ones exist.  This is captured by the following theorem.

\begin{theorem}\label{thm:ACCEPTABLE.EXIST}
If $\zeta$ is unimodular, then there are acceptable crosscuts of arbitrarily small diameter that recognizably bound the value of $\phi$ on $\zeta$.
\end{theorem}

\begin{proof}
Suppose $C$ recognizably bounds the value of $\phi$ on $\zeta$.  By Theorem \ref{thm:EXISTENCE}, the diameter of $C$ can be made as small as desired.  

If $\phi(\zeta)$ is not a vertex, and if the diameter of $C$ is small enough, then its endpoints are not vertices and belong to the same significant constituent of the boundary of $D$; thus they are acceptably placed.  So, suppose $\phi(\zeta)$ is a vertex. 
 If $\phi(\zeta)$ is not $\nu_0$, and if the diameter of $C$ is small enough, then the endpoints of $C$ are not vertices and belong to adjacent line segments in the boundary of $D$ and so are acceptably placed.  
 
 So, suppose $\phi(\zeta) = \nu_0$.  We note that $\phi(\zeta') \neq \nu_0$ for all unimodular $\zeta'$ sufficiently close to $\zeta$ except $\zeta$ itself.  In addition, $\zeta$ is an accumulation point of the pre-image of $\phi$ on the unit interval.  By following the argument in the proof of Theorem 3.4 of \cite{McNicholl.2011.b}, it follows that for each $\epsilon > 0$, there is a crosscut $C$ that recognizably bounds the value of $\phi$ on $\zeta$ and so that one endpoint of $C$ lies on $[0, i]$ and the other lies in the unit interval.  We can also ensure that neither endpoint of $C$ is a vertex.  Thus, the endpoints of $C$ are acceptably placed.  By making the diameter of $C$ small enough, we can also ensure that $C$ is acceptable.
 \end{proof}
 
 \subsection{Approximate world}\label{subsec:APPROX}
 
 In order to apply the results in Subsection \ref{subset:CLASSICAL}, we need a method for generating approximate crosscuts that approximate acceptable crosscuts.  This leads to the following three definitions and theorem.  
 
 \begin{definition}\label{def:CONSERVATIVELY}
 Let $U$ be a subset of the plane.  We say that $U$ \emph{conservatively intersects} the boundary of $D$ if there is a significant constituent of the boundary of $D$, $\sigma$, so that $U \cap \sigma \neq \emptyset$ and so that $\overline{U}$ contains no vertex of the boundary of $D$ and no point of any significant constituent of the boundary of $D$ other than $\sigma$.
 \end{definition}
 
 \begin{definition}\label{def:SET.ACCEPTABLY.PLACED}
 Let $U,V$ be subsets of the plane.  We say that $U,V$ are \emph{acceptably placed} if $p,q$ are acceptably placed whenever $p \in U \cap \partial D$ and $q \in V \cap \partial D$.
 \end{definition}
 
 \begin{definition}\label{def:ALPHA.*}\label{def:ALPHA.STAR}
 Suppose $\mathcal{A}$ is a set of approximate crosscuts.  Let $\mathcal{A}^*$ consist of all $(w_1, \ldots, w_n) \in \mathcal{A}$ so that $w_1$ and $w_n$ conservatively intersect the boundary of $D$, $w_1, w_n$ are acceptably placed, and so that 
 \[
 (1 + \sqrt{2})\diam(\bigcup_j \overline{w_j}) < |\phi(0) - \phi(\zeta/2)|
 \]
 for all unimodular $\zeta$. 
 \end{definition}
 
 It follows that \emph{every} arc approximated by an approximate crosscut in $\mathcal{A}^*$ is acceptable.  
 
 \begin{theorem}\label{thm:ALPHA.*}
 If $\mathcal{A}$ describes $\mathcal{C}_{(s_0, r_0, \zeta)}$, then $\mathcal{A}^*$ describes the set of all acceptable crosscuts in $\mathcal{C}_{(s_0, r_0, \zeta)}$.  
 \end{theorem}
 
 \begin{proof}
 Since $\mathcal{A}$ describes $\mathcal{C}_{(s_0, r_0, \zeta)}$, it follows that each approximate crosscut in $\mathcal{A}^*$ approximates an acceptable crosscut in $\mathcal{C}_{(s_0, r_0, \zeta)}$.  So, let $C$ be an acceptable crosscut in $\mathcal{C}_{(s_0, r_0, \zeta)}$, and let $\epsilon$ be a positive number.  We need to show that $C$ is approximated by an approximate crosscut in $\mathcal{A}^*$ with error smaller than $\epsilon$.  So, fix a positive number $\delta < \epsilon$.  Since $\mathcal{A}$ describes $\mathcal{C}_{(s_0, r_0, \zeta)}$, some $(w_1, \ldots, w_n) \in \mathcal{A}$ approximates $C$ with error smaller than $\delta$.  Since $C$ is acceptable, if $\delta$ is small enough, then $w_1$ and $w_n$ conservatively intersect the boundary of $D$.  In addition, since $C$ is acceptable, if $\delta$ is small enough, then 
 \[
 (1 + \sqrt{2})\diam(\bigcup_j \overline{w_j}) < |\phi(0) - \phi(\zeta/2)|
 \]
 for all unimodular $\zeta$.  Finally, since $C$ is acceptable, if $w_1$ and $w_n$ conservatively intersect the boundary of $D$, then it follows from Lemma \ref{lm:ACCEPTABLY.PLACED} that $w_1$, $w_n$ are acceptably placed.  Thus, $(w_1, \ldots, w_n) \in \mathcal{A}^*$ if $\delta$ is small enough.  Thus, there is an approximate crosscut in $\mathcal{A}^*$ that approximates $C$ with error smaller than $\epsilon$.
 \end{proof}
 
 \subsection{Computable world}\label{subsec:COMPUTABLE}
 
  Our first goal is to show that $\mathcal{A}^*$ is computably enumerable whenever $\mathcal{A}$ is. This is accomplished via the following two lemmas.   Let $\sigma_0$, $\sigma_1$, $\ldots$ label the significant constituents of the boundary of $D$ in clockwise order so that $\sigma_0 = [0, i]$.

 \begin{lemma}\label{lm:SIGMA.CE}
 For each $k$, $\sigma_k$ is computably closed uniformly in $k$.
 \end{lemma}
 
 \begin{proof}
 If $\sigma_k$ is neither a tent nor a spike, then this is obvious.  So, suppose $\sigma_k$ is either a tent or a spike.  Then, there exists unique $j \geq 1$ so that $\sigma_k = [\nu_{3j + 1}, \nu_{3j + 2}] \cup [\nu_{3j + 2}, \nu_{3(j+1)}]$.  
In addition, by construction, $\nu_{3j+1} = \nu_{3(j+1)}$ if and only if $j \not \in A$.  Furthermore, 
if $j \in A_{\mbox{at $s$}}$, then $\nu_{3j + 1} = 2^{-(j+1)} + 2^{-(j + 3 + s)}$ and $\nu_{3(j+1)} = 2^{-(j+1)} - 2^{-(j + 3 + s)}$.  
 
Let $R = (a,b) \times (c,d)$.  It follows that $R \cap \sigma_k \neq \emptyset$ if and only if 
there exists $s$ so that $j \in A_{\mbox{at $s$}}$ and 
\[
R \cap ([2^{-(j+1)} - 2^{-(j + 3 + s)}, \nu_{3j+2}] \cup [2^{-(j+1)} + 2^{-(j + 3 + s)}, \nu_{3j+2}]) \neq \emptyset
\]
or so that $j \not \in A_s$ and $(a,b)$ includes one of the intervals $[2^{-(j+1)} - 2^{-(j + 3 + s)}, 2^{-(j+1)}]$, $[2^{-(j+1)}, 2^{-(j+1)} + 2^{-(j + 3 + s)}]$.  Hence, the set of all open rational rectangles that contain a point of $\sigma_k$ is computably enumerable; that is, $\sigma_k$ is computably closed.
 \end{proof}
 
 \begin{lemma}\label{lm:SIGMA.CO.CE}
 For each $k \in \N$, the complement of $\overline{\partial D - \sigma_k}$ is computably open uniformly in $k$.
 \end{lemma}
 
 \begin{proof}
 For each $j \geq 1$, let 
 \[
 m_j = 2^{-(j+3)} + 2^{-(j+5)} + 2^{-(j + 2)} - 2^{-(j+4)}.
 \]
 It follows that $m_j$ is between $\nu_{3j + 4}$ and $\nu_{3(j+1)}$ for all $j \geq 1$.  
 For each $j$, let 
 \[
 R_j = (-2^{-j}, m_j) \times (-2^{-j}, m_j).
 \]
 Thus, $\nu_0 \in R_j$ and $\nu_k \in R_j$ when $k > 3(j+1)$.  So, $R_j$ includes all but finitely many significant constituents of the boundary of $D$. 
 
 Let $R$ be a closed rational rectangle.  It follows that $R$ contains no point of $\overline{\partial D - \sigma_k}$ if an only if there exists $j \geq 1$ so that $R_j \cap R = \emptyset$ and $\sigma' \cap R = \emptyset$ whenever $\sigma'$ is a significant constituent of the boundary of $D$ such that $\sigma_k \neq \sigma'$ and $\sigma' \not \subseteq R_j$.  It follows that the complement of $\overline{\partial D - \sigma_k}$ is computably open.
 \end{proof}
 
 \begin{theorem}\label{thm:CE}
 If $\mathcal{A}$ is a computably enumerable set of approximate crosscuts of $D$, then $\mathcal{A}^*$ is computably enumerable uniformly in $\mathcal{A}$.
 \end{theorem}
 
 \begin{proof}
 Suppose $\mathcal{A}$ is computably enumerable.   It follows from Lemmas \ref{lm:SIGMA.CE} and \ref{lm:SIGMA.CO.CE} that the set of all wads that conservatively intersect the boundary of $D$ is computably enumerable.  The set of all $k,k'$ so that $p,q$ are acceptably placed whenever $p \in \sigma_k$ and $q \in \sigma_{k'}$ is computable.  In addition, the set of all $(w_1, \ldots, w_n)$ so that 
 \[
 \diam(\bigcup_j \overline{w_j}) < (1 + \sqrt{2}) |\phi(0) - \phi(\zeta/2)|
 \]
 for all unimodular $\zeta$ is computably enumerable.  
 It follows that $\mathcal{A}^*$ is computably enumerable uniformly in $\mathcal{A}$.
 \end{proof}
 
 \begin{lemma}\label{lm:Rk}
 From $k \in \N$, it is possible to uniformly compute a finite set of open rational rectangles $\mathcal{R}_k$ that covers the unit circle and so that $|\phi(z_1) - \phi(z_2)| < 2^{-k}$ whenever $R \in \mathcal{R}_k$ and $z_1, z_2 \in R \cap \overline{\D}$.  
 \end{lemma}
 
 \begin{proof}
  For each rational number $\theta$, let $\zeta_\theta = \exp(\theta i )$.  
  By Theorem \ref{thm:DESCRIBES.CE}, $\mathcal{A}_{(s_0, r_0, \zeta_\theta)}$ is computably enumerable uniformly in $(s_0, r_0, \theta)$.  So, by Theorem \ref{thm:ALPHA.*} and \ref{thm:CE}, $\mathcal{A}_{(s_0, r_0, \zeta_\theta)}^*$ describes the set of all acceptable crosscuts in $\mathcal{C}_{(s_0, r_0, \zeta_\theta)}$ and is computably enumerable uniformly in $s_0, r_0, \theta$.  

 The proof now mimics the proof of Lemma 6.3 of \cite{McNicholl.2011.b}.   Let $\mathcal{R}_k'$ be the set of all open rational rectangles $R$ for which there exist $s_0, r_0, \theta \in \Q$ and an acceptable $C \in \mathcal{C}_{(s_0, r_0, \zeta_\theta)}$ such that $\zeta_\theta \in R \subseteq D_{r_0}(\zeta_\theta)$ and such that the diameter of $C$ is smaller than 
 $(1 + \sqrt{2})^{-1} 2^{-k}$.  It follows that $\mathcal{R}_k'$ is computably enumerable uniformly in $k$.  It then follows from Theorem \ref{thm:WITNESSES} that $|\phi(z_1) - \phi(z_2)| < 2^{-k}$ whenever $z_1, z_2 \in R \cap \overline{\D}$ and $R \in \mathcal{R}_k'$.  
 
 We claim that $\mathcal{R}_k'$ covers the unit circle.  For, suppose $|\zeta| = 1$.  By Theorem \ref{thm:ALPHA.*}, there is an acceptable crosscut whose diameter is smaller than $(1+ \sqrt{2})^{-1}2^{-k}$ and that recognizably bounds the value of $\phi$ on $\zeta$.  Let $(s_0, r_0)$ witness that $C$ recognizably bounds the value of $\phi$ on $\zeta$.  Since $\mathcal{A}_{(s_0, r_0, \zeta)}^*$ describes $\mathcal{C}_{(s_0, r_0, \zeta)}$, $C$ is approximated by an approximate crosscut $(w_1, \ldots, w_n) \in \mathcal{A}_{(s_0, r_0, \zeta)}^*$ with the property that the diameter of $\overline{w_1} \cup \ldots \cup \overline{w_n}$ is smaller than $(1 + \sqrt{2})^{-1} 2^{-k}$.  
 By Lemma \ref{lm:UNIMODULAR}, there is a closed rational rectangle $R \subseteq D_{r_0}(\zeta)$ such that $\zeta$ belongs to the interior of $R$ and such that for all $\zeta' \in R \cap \partial \D$, $(w_1, \ldots, w_n)$ approximates a crosscut $C'$ such that $(s_0, r_0)$ witnesses that $C'$ recognizably bounds the value of $\phi$ on $\zeta'$.  Since $(w_1, \ldots, w_n) \in \mathcal{A}_{(s_0, r_0, \zeta)}^*$, as remarked after Definition \ref{def:ALPHA.STAR}, every approximate crosscut approximated by $(w_1, \ldots, w_n)$ is acceptable.  Thus, $C'$ is acceptable.  The interior of $R$ contains a point of the form $\zeta_\theta$ for some $\theta \in \Q$.  Since $R$ is closed, if $\zeta_\theta$ is close enough to $\zeta$, then $R \subseteq D_{r_0}(\zeta_\theta)$.  Thus, the interior of $R$ belongs to $\mathcal{R}_k'$.  Hence, the unit circle is covered by $\mathcal{R}_k'$.  
 
 To compute $\mathcal{R}_k$, enumerate $\mathcal{R}_k'$ until the unit circle is covered.
 \end{proof}
 
 The following now follows as in the proof of Theorem 2.8 of \cite{McNicholl.2011.b}.
 
 \begin{theorem}
 The boundary extension of $\phi$ is strongly computable on the unit circle.
 \end{theorem}
 
 So, it follows from Proposition \ref{prop:STRONGLY.COMP} that the boundary extension of $\phi$ is computable.  To complete the proof of Theorem \ref{thm:MAIN}, it only remains to show that $D$ does not have a computable boundary connectivity function.  In fact, we can show slightly more.
 
 \begin{theorem}\label{thm:NOT.ELC}
If $g$ is a boundary connectivity function for $D$, then $A$ is Turing reducible to $g$.  
\end{theorem}

\begin{proof}
Let $g$ be a boundary connectivity function for $D$.  We can assume $g$ is increasing.  Write $n \in A_s$ if $n$ enters $A$ at some $s' \leq s$.

We claim that $n \not \in A$ if $n \not \in A_{g(n+2)}$.  For, by way of contradiction, suppose that $n \in A - A_{g(n+2)}$.  Let $s$ be the stage at which $n$ enters $A$.  Thus, $s > g(n+2)$.  It follows from the construction that the distance from $\nu_{3n + 4}$ to $\nu_{3n + 6}$ is $2^{-(n+2 + s)}$ which in turn is not larger than $2^{-g(n+2)}$.  However, by construction, the shortest arc in $X$ from $\nu_{3n+4}$ to $\nu_{3n + 6}$ consists of the line segment from $\nu_{3n + 4}$ to $\nu_{3n + 5}$ and the line segment from $\nu_{3n + 5}$ to $\nu_{3n +6}$.  The diameter of each of these segments is at least $2^{-(n+1)}$, and this is a contradiction. 
  Thus, $n \not \in A$ if  $n \not \in A_{g(n+2)}$.  
  
It now follows that $A$ is Turing reducible to $g$.
\end{proof}

\section{Acknowledgement} My deepest thanks to Justin Peters for helpful conversations.  

%

\begin{thebibliography}{10}

\bibitem{Bishop.Bridges.1985}
Errett Bishop and Douglas Bridges, \emph{Constructive analysis}, Grundlehren
  der Mathematischen Wissenschaften [Fundamental Principles of Mathematical
  Sciences], vol. 279, Springer-Verlag, Berlin, 1985.

\bibitem{Brattka.Weihrauch.1999}
Vasco Brattka and Klaus Weihrauch, \emph{Computability on subsets of
  {E}uclidean space. {I}. {C}losed and compact subsets}, Theoret. Comput. Sci.
  \textbf{219} (1999), no.~1-2, 65--93, Computability and complexity in
  analysis (Castle Dagstuhl, 1997).

\bibitem{Braverman.Cook.2006}
M.~Braverman and S.~Cook, \emph{Computing over the reals: foundations for
  scientific computing}, Notices of the American Mathematical Society
  \textbf{53} (2006), no.~3, 318--329.

\bibitem{Cooper.2004}
S.~Barry Cooper, \emph{Computability theory}, Chapman \& Hall/CRC, Boca Raton,
  FL, 2004.

\bibitem{Couch.Daniel.McNicholl.2012}
P.J. Couch, B.D. Daniel, and T.H. McNicholl, \emph{Computing space-filling
  curves}, Theory of Computing Systems \textbf{50} (2012), no.~2, 370--386.

\bibitem{Daniel.McNicholl.2012}
D.~Daniel and T.H. McNicholl, \emph{Effective local connectivity properties},
  Theory of Computing Systems \textbf{50} (2012), no.~4, 621 -- 640.

\bibitem{Grzegorczyk.1957}
A.~Grzegorczyk, \emph{On the definitions of computable real continuous
  functions}, Fund. Math. \textbf{44} (1957), 61--71.

\bibitem{Hertling.1999}
P.~Hertling, \emph{An effective {Riemann Mapping Theorem}}, Theoretical
  Computer Science \textbf{219} (1999), 225 -- 265.

\bibitem{Lacombe.1955.a}
Daniel Lacombe, \emph{Extension de la notion de fonction r\'ecursive aux
  fonctions d'une ou plusieurs variables r\'eelles. {I}}, C. R. Acad. Sci.
  Paris \textbf{240} (1955), 2478--2480. \MR{0072079 (17,225d)}

\bibitem{Lacombe.1955.b}
\bysame, \emph{Extension de la notion de fonction r\'ecursive aux fonctions
  d'une ou plusieurs variables r\'eelles. {II}, {III}}, C. R. Acad. Sci. Paris
  \textbf{241} (1955), 13--14, 151--153. \MR{0072080 (17,225e)}

\bibitem{McNicholl.2011.b}
T.H. McNicholl, \emph{Computing boundary extensions of conformal maps},
  Submitted. Preprint available at http://arxiv.org/abs/1110.5271v1.

\bibitem{Pommerenke.1992}
Ch. Pommerenke, \emph{Boundary behaviour of conformal maps}, Grundlehren der
  Mathematischen Wissenschaften [Fundamental Principles of Mathematical
  Sciences], vol. 299, Springer-Verlag, Berlin, 1992.

\bibitem{Pour-El.Richards.1989}
Marian~B. Pour-El and J.~Ian Richards, \emph{Computability in analysis and
  physics}, Perspectives in Mathematical Logic, Springer-Verlag, Berlin, 1989.

\bibitem{Turing.1937}
A.~M. Turing, \emph{On {C}omputable {N}umbers, with an {A}pplication to the
  {E}ntscheidungsproblem. {A} {C}orrection}, Proc. London Math. Soc.
  \textbf{S2-43} (1937), no.~6, 544.

\bibitem{Weihrauch.2000}
Klaus Weihrauch, \emph{Computable analysis}, Texts in Theoretical Computer
  Science. An EATCS Series, Springer-Verlag, Berlin, 2000.

\end{thebibliography}
\def\cprime{$'$}
\providecommand{\bysame}{\leavevmode\hbox to3em{\hrulefill}\thinspace}
\providecommand{\MR}{\relax\ifhmode\unskip\space\fi MR }
\providecommand{\MRhref}[2]{%
  \href{http://www.ams.org/mathscinet-getitem?mr=#1}{#2}
}
\providecommand{\href}[2]{#2}

\end{document}